\documentclass[10pt,a4paper]{article}
\usepackage[ansinew]{inputenc}
\usepackage{amsmath,amssymb,amsthm}
\newcommand{\erz}[1]{\langle #1 \rangle}

\newtheorem{Th}{Theorem}[section]
\newtheorem{Le}[Th]{Lemma}
\newtheorem{Res}[Th]{Result}

\newcommand{\gauss}[2]{{#1\brack #2}}

\title{A note on Erd\H os-Ko-Rado sets of generators in Hermitian polar spaces}

\author{Klaus Metsch\thanks{Justus-Liebig-Universit\"{a}t, Mathematisches Institut,
Arndtstra{\ss}e 2, D-35392 Gie{\ss}en, and Department of Pure Mathematics and Computer Algebra, Ghent University, Krijgslaan 281-S22, 9000 Gent, Belgium}  }

\date{\today}

\begin{document}

\thispagestyle{empty}
\parindent0ex
\parskip2ex

\maketitle

\setcounter{page}{1}
\begin{abstract}
The size of the largest Erd\H os-Ko-Rado set of generators in the finite classical polar space is known for all polar spaces except for $H(2d-1,q^2)$ when $d\ge 5$ is odd. We improve the known upper bound in this remaining case by using a variant of the famous Hoffman's bound.
\end{abstract}



\section{Introduction}

An Erd\H os-Ko-Rado set of generators in a finite classical polar space is a set of generators of the polar space that have mutually non-trivial intersection. The largest size of an Erd\H os-Ko-Rado set of generators in a finite classical polar space was determined in \cite{Pepe} for all finite classical polar spaces except for the hermitian polar space $H(2d-1,q^2)$ of odd rank $d\ge 5$. Here rank means vector space rank and not projective dimension. The best known upper bound for this remaining case was proved in \cite{Ihringer&Metsch}. The idea of the proof was to formulate a linear optimization problem whose solution gives an upper bound. This idea goes back to Delsarte and uses the primitive idempotents of the associations scheme related to set of generators of a polar space, see \cite{Brouwer}. In \cite{Ihringer&Metsch} we were however not able to determine the optimal solution of the optimization problem. Using a slightly different approach, the previous bound can be improved as follows.

\begin{Th}\label{main}
If $S$ is an Erd\H os-Ko-Rado set of generators of $H(2d-1,q^2)$, $d\ge 5$ odd, then
\begin{align}\label{mainresult}
|S|\le ((q^2+q+1)q^{2d-3}+1)\prod_{i=1\atop 2i\not=d\pm1}^{d-1}(q^{2i-1}+1).
\end{align}
\end{Th}

{\bf Remarks}  1. The set consisting of all generators of $H(2d-1,q^2)$, $d$ odd, on a point has size roughly $q^{d^2}$ whereas the bound given in the theorem has size roughly $q^{d^2+1}$.

2. It can be shown that equality can not occur in the theorem. At the end of Section 2 we sketch a proof of this fact.

3. For small $d$ it can be checked by computer that the given bound is also the solution of the optimization problem mentioned above. I guess this is true for all $d$, but I did not try to show this.

\section{Proof of the theorem}

Consider the graph whose vertices are the generators of the hermitian polar space $H(2d-1,q^2)$ of odd rank $d\ge 3$. Let $N$ be the number of generators and number them as $G_1,\dots,G_N$. For $0\le i\le d$, let $A_i$ be the real symmetric $(N\times N)$-matrix whose $(r,s)$-entry is $1$, if $G_r\cap G_s$ has rank $d-i$, and $0$ otherwise. These real matrices are symmetric and commute pairwise, so they can simultaneously be diagonalized. It is known that there are exactly $d+1$ common eigenspaces $V_0,\dots,V_d$ of these matrices. Also one of the eigenspaces is $\erz{j}$ where $j$ is the all one vector of length $N$. We choose notation so that $V_0=\erz{j}$. If $P_{i,j}$ denotes the eigenvalue of $A_j$ on $V_i$, then with a suitable ordering of the eigenspaces we have, see \cite{Vanhove}, Theorem 4.3.6
\begin{eqnarray*}
P_{i,d}&=&(-1)^i q^{(d-i)^2+i(i-1)},
\\
P_{i,d-2}&=&\sum_{u=0}^2(-1)^{i+u}\gauss{d-i}{2-u}\gauss{i}{u}q^{(d-2+u-i)^2+(i-u)(i-u-1)}.
\end{eqnarray*}
We want to apply Hoffman's bound (see blow) to the generalized adjacency matrix $A:=A_d-fA_{d-2}$ where we use for $f$ the value for which the smallest eigenvalue of $A$ is as large as possible. The eigenvalues for $A$ are of course $P_{i,d}-fP_{i,d-2}$, $i=0,\dots,d$. An investigation shows that the best choice for $f$ is when $P_{d,d}-fP_{d,d-2}=P_{1,d}-fP_{1,d-2}$, which results in the following definition.
\[
f:=\frac{(q^{d-1}-1)q^{4(d-2)}}{\gauss{d-1}{1}q^{2d-5}-\gauss{d-1}{2}+\gauss{d}{2}q^{d-3}}
\]
A direct calculation shows that

\begin{eqnarray*}
f&=&
\frac{(q^{2d}-1)(q^{d-1}-1)q^{4(d-2)}}{\gauss{d}{2}(q^{d-2}+1)(q^{2d-1}-q^{d-2}-q^{d-3}+1)}
\\
&=&\frac{(q^2-1)(q^4-1)(q^{d-1}-1)q^{4(d-2)}}{(q^{2d-2}-1)(q^{d-2}+1)(q^{2d-1}-q^{d-2}-q^{d-3}+1)}
\\
&<&q^2-1.
\end{eqnarray*}

\begin{Le}
The matrix $A$ has constant row sum
\[
K=q^{d^2}-f\gauss{d}{2}q^{(d-2)^2}>0.
\]
\end{Le}
\begin{proof}
The row sum of $A$ is the eigenvalue of $A$ on the eigenspace $\erz{j}$. Using $f<q^2-1$, it follows that $K>0$.
\end{proof}

\begin{Le}\label{smallesteigenvalue}
The smallest eigenvalue of $A$ is
\[
\lambda:=-q^{d(d-1)}+f\gauss{d}{2}q^{(d-2)(d-3)}
\]
and we have $\lambda<-q^{d^2-2d+2}$.
\end{Le}
\begin{proof}
It follows from the list of eigenvalues that $\lambda$ is the eigenvalue of $A=A_d-fA_{d-2}$ on the eigenspace $V_d$. Also, the way we determined $f$ shows that the eigenvalue of $A$ on $V_1$ is also $\lambda$. A straightforward calculation shows that
\[
\lambda=-\frac{(q+1)(q^{2d}-q^{2d-3}+q-1)q^{d^2-d-2}}
{(q^{d-2}+1)(q^{2d-1}-q^{d-2}-q^{d-3}+1)}
\]
and, using this expression, it is easy to see that $\lambda<-q^{d^2-2d+2}$.

The eigenvalue of $A$ on $V_0$ is $K$ and the previous lemma shows that $K>0$. For $1\le i\le d-1$, the eigenvalue of $A$ on $V_i$ is $P_{i,d}-fP_{i,d-2}$, and we show in the remaining part of the proof that this eigenvalue is larger than $\lambda$.

First consider the case when $i$ is odd. Then in the above formula for $P_{i,d-2}$ as a sum over $u\in\{0,1,2\}$, only the term corresponding to $u=1$ is positive. Hence
\begin{eqnarray*}
P_{i,d-2}&\le&\gauss{d-i}{1}\gauss{i}{1}q^{(d-1-i)^2+(i-1)(i-2)}
\\&\le&\frac{q^{(d-i)^2+i^2-i+3}}{(q^2-1)^2}.
\end{eqnarray*}
Using $f<q^2-1$, we obtain the following bound for the eigenvalue of $A$ on $V_i$.
\begin{eqnarray*}
P_i-fP_{i,d-2}&\ge&P_i-f\cdot \frac{q^{(d-i)^2+i^2-i+3}}{(q^2-1)^2}
\\
&\ge&-q^{(d-i)^2+i^2-i}-\frac{q^{(d-i)^2+i^2-i+3}}{(q^2-1)}
\\&\ge&-q^{d^2-2d+2}>\lambda.
\end{eqnarray*}
Here we have used that $3\le i\le d-2$ (since $i$ is odd).

If $i$ is even, then $P_{i,d}>0$ and it is not difficult to see that $P_{i,d-2}<0$. In this case the eigenvalue $P_i-fP_{i,d-2}$ of $A$ on $V_i$ is positive.
\end{proof}

\begin{Le}\label{lemmabound}
If $N$ is the number of generators of $H(2d-1,q^2)$, then
\[
\frac{-\lambda N}{K-\lambda}=((q^2+q+1)q^{2d-3}+1)\prod_{i=1\atop 2i\not=d\pm1}^{d-1}(q^{2i-1}+1).
\]
\end{Le}
\begin{proof}
We denote by $f_1$ and $f_2$ the nominator and denominator in the definition of $f$. We have
\begin{eqnarray*}
\frac{-\lambda}{K-\lambda}
&=&\frac{q^{d(d-1)}f_2-(q^{d-1}-1)q^{4(d-2)}\gauss{d}{2}q^{(d-2)(d-3)}}
{(q^{d^2}+q^{d(d-1)})f_2-(q^{d-1}-1)q^{4(d-2)}\gauss{d}{2}(q^{(d-1)^2}+q^{(d-2)(d-3)})}
\\
&=&\frac{q^2f_2-(q^{d-1}-1)\gauss{d}{2}}
{q^2(q^d+1)f_2-(q^{d-1}-1)\gauss{d}{2}(q^{d-2}+1)}\ .
\end{eqnarray*}
An easy calculation gives
\[
f_2(q^{2d}-1)=\gauss{d}{2}g
\]
where
\[
g=(q^4-1)q^{2d-5}-(q^{2d-4}-1)+(q^{2d}-1)q^{d-3}.
\]
Hence
\begin{eqnarray*}
\frac{-\lambda}{K-\lambda}
&=&\frac{q^2g-(q^{d-1}-1)(q^{2d}-1)}
{q^2(q^d+1)g-(q^{d-1}-1)(q^{2d}-1)(q^{d-2}+1)}\ .
\end{eqnarray*}
Thus
\begin{eqnarray*}
\frac{-\lambda(q^d+1)}{K-\lambda}
&=&\frac{q^2g-(q^{d-1}-1)(q^d-1)}
{q^2g-(q^{d-1}-1)(q^d-1)(q^{d-2}+1)}
\\
&=&1+\frac{(q^{d-1}-1)(q^d-1)q^{d-2}(1-q^2)}
{q^2g-(q^{d-1}-1)(q^d-1)(q^{d-2}+1)}
\\
&=&1+\frac{(q^{d-1}-1)(q^d-1)q^{d-2}(1-q^2)}
{(q^2-1)(q^{2d-1}+1)(q^{d-2}+1)}
\\
&=&1-\frac{(q^{d-1}-1)(q^d-1)q^{d-2}}
{(q^{2d-1}+1)(q^{d-2}+1)}
\\
&=&\frac{(q^2+q+1)q^{2d-3}+1}
{(q^{2d-1}+1)(q^{d-2}+1)}.
\end{eqnarray*}
As $N=\prod_{i=1}^d(q^{2i-1}+1)$, the assertion follows.
\end{proof}

Let $G$ be a simple and non-empty graph with $N$ vertices $v_1,\dots,v_N$. A real symmetric $N\times N$ matrix $A$ with diagonal entries zero is called an {\it extended weight matrix} of $G$ if $A_{rs}\le 0$ whenever $r\not=s$ and $\{v_r,v_s\}$ is not an edge of the graph $G$, and if $A_{rs}\not=0$ for at least one edge $\{v_r,v_s\}$ of the graph. It is called a $K$-{\it regular} extended weight matrix, if it has in addition constant row sum $K$. The following result appeared in various forms in the literature, we present it in the form of Corollary 3.3 in \cite{Elzinga} but it already appeared in Lemma 6.1 of \cite{Godsil} when applied to the matrix $A-\lambda I$, and it was also mentioned in \cite{Luz}. For later application, we scetch the easy proof. Here $I$ stand for the $N\times N$ identity matrix and $J$ for the all-one matrix of the same size.

\begin{Res}\label{Resratio}
Let $\Gamma$ be a finite simple and non-empty graph with $N$ vertices and suppose that $A$ is a $K$-regular generalized weight matrix of $G$ with least eigenvalue $\lambda$. Then every independent set $S$ of $G$ satisfies
\begin{align*}
|S|\cdot (K+|\lambda|)\le |\lambda|\cdot N.
\end{align*}
\end{Res}

From Result \ref{Resratio} applied to $A=A_d-fA_{d-2}$ and from Lemma \ref{lemmabound}, we find that an EKR set $S$ of $H(2d-1,q^2)$ satisfies the inequality (\ref{mainresult}) of Theorem \ref{main}. This completes the proof of Theorem \ref{main}.

We remark hat equality in (\ref{mainresult}) is impossible and sketch a proof.
Suppose that $|S|$ satisfies the bound (\ref{mainresult}) with equality. Then the standard proof of Result \ref{Resratio} gives information on the characteristic vector $\chi$ of $S$, in fact, it must lie in the span of the all-one-vector $j$ and the the eigenspace of $A$ for the eigenvalue $\lambda$. Our arguments show that this eigenspace of $A$ is $V_1+V_d$, hence $\chi=\frac{|S|}{N}j+v_1+v_d$ with $v_1\in V_1$ and $v_d\in V_d$ (here we use $j^\top v_1=j^\top v_d=0$, $j^\top j=N$ and $j^\top\chi=|S|$). The vectors $v_1$ and $v_d$ are also eigenvectors of $A_1,\dots,A_d$ and the eigenvalues are known. Since $S$ is an Erd\H os-Ko-Rado set, the entries of $A_d\chi$ corresponding to elements of $S$ are zero. This gives a linear equation for the entries $a_1$ and $a_d$ of $v_1$ and $v_d$ corresponding to some element of $S$. A second linear equation comes from $\chi=\frac{|S|}{N}j+v_1+v_d$, namely $1=\frac{|S|}{N}+a_1+a_d$. The two
equations are linearly independent, so $a_1$ and $a_d$ can be calculated and are of course independent of the element of $S$. With this information the entries of $A_i\chi$ corresponding to elements of $S$ can be calculated for all $i$, which gives the number of elements of $S$ that meet a given element of $S$ in dimension $d-i$. It turns out that not all these numbers are integers, which is the desired contradiction. The same argument was used in the last section of \cite{MetschEKRn}.

\bibliography{bibdatei}

\end{document}